\documentclass[11pt,letterpaper]{amsart}
\usepackage{amsfonts,amsmath,amsthm,amssymb}

\usepackage{graphicx}
\usepackage{setspace}
\usepackage{thmtools}

\usepackage{dsfont}
\usepackage{grffile}

\usepackage [english]{babel}
\usepackage [autostyle, english = american]{csquotes}
\MakeOuterQuote{"}

\numberwithin{equation}{section}

\usepackage{xcolor}
\usepackage{hyperref}
\hypersetup{
    colorlinks = true,
    linkcolor = {blue}, urlcolor={blue}
}

\definecolor{c1}{rgb}{0,0,1} 
\hypersetup{
    linkcolor= {c1}, 
    citecolor={c1}, 
    urlcolor={c1} 
}

\newtheorem{thm}{Theorem}

\newtheorem{lemma}{Lemma}

\numberwithin{equation}{section}
\numberwithin{lemma}{section}
\numberwithin{cor}{section}
\numberwithin{prop}{section}

\begin{document}

{
 \title[Fractional part correlations and the M{\"o}bius function]{Fractional part correlations and the M{\"o}bius function}
  \author{Gordon Chavez}

\date{}
  \maketitle
}

\begin{abstract}
We show that 
$$
\sum_{n\neq m}\frac{\mu(n)\mu(m)}{nm}E_{X}\left(\{nx\}\{mx\}\right)=-\frac{9}{2\pi^{2}}+O\left(\frac{1}{X}\right),
$$
where $x$ is uniformly distributed in $[0,X]$ with $X\in \mathbb{N}$, $E_{X}(.)$ denotes the expected value, $\mu(.)$ denotes the M{\"o}bius function, and $\{.\}$ denotes the fractional part function. 
\end{abstract}

\setcounter{tocdepth}{1}

\tableofcontents

\section{Introduction}

For a given real number $x$, the floor function $\lfloor x\rfloor$ outputs the largest integer less than or equal to $x$, i.e., 
\begin{equation}
\lfloor x \rfloor=\textnormal{max}\left(k \in \mathbb{Z}\hspace{.05cm}|\hspace{.05cm} k\leq x\right). \label{floor}
\end{equation}
The fractional part function $\{x\}$ may then be defined in terms of (\ref{floor}) as
\begin{equation}
\{x\}=x-\lfloor x \rfloor. \label{frac}
\end{equation}
These simple functions find a large number of applications in mathematics, computer science, and related fields. It is natural to consider the correlations between $\{nx\}$ and $\{mx\}$, where $n$ and $m$ are fixed positive integers and $x$ varies over the real line. These correlations are defined using the expected value, which itself is defined here as the integral
\begin{equation}
E_{X}\left(f(x)\right)=\frac{1}{X}\int_{0}^{X}f(x)dx, \hspace{.25cm} X\in \mathbb{N}.\label{expected value}
\end{equation}
The expected value is thus evaluated over an integer-valued range. The correlations we primarily consider are then $E_{X}\left(\{nx\}\{mx\}\right)$. Numerical evidence shows that this correlation function has nontrivial, number theoretic structure. As an example, in Figure \ref{fig0} we depict an approximation of $E_{100}\left(\{nx\}\{6x\}\right)$. There are several observations to note: 1. The correlation jumps if $n$ is a multiple of $6$ or vice versa. 2. The correlation is minimal at $n$ such that $n$ and $6$ are coprime. 3. The correlation is higher for $n$ such that $n$ and $6$ share a prime factor and further, this correlation is higher still for $n$ such that $n$ has a repeated prime factor that is shared with $6$. These properties are observed numerically for integers other than $6$ as well. An explicit formula for $E_{X}\left(\{nx\}\{mx\}\right)$ is currently unknown, but its number theoretic properties merit further study. To that end, this paper investigates connections between this correlation function and the well-known M{\"o}bius function.

\begin{figure}
\centering
\includegraphics[width=\linewidth]{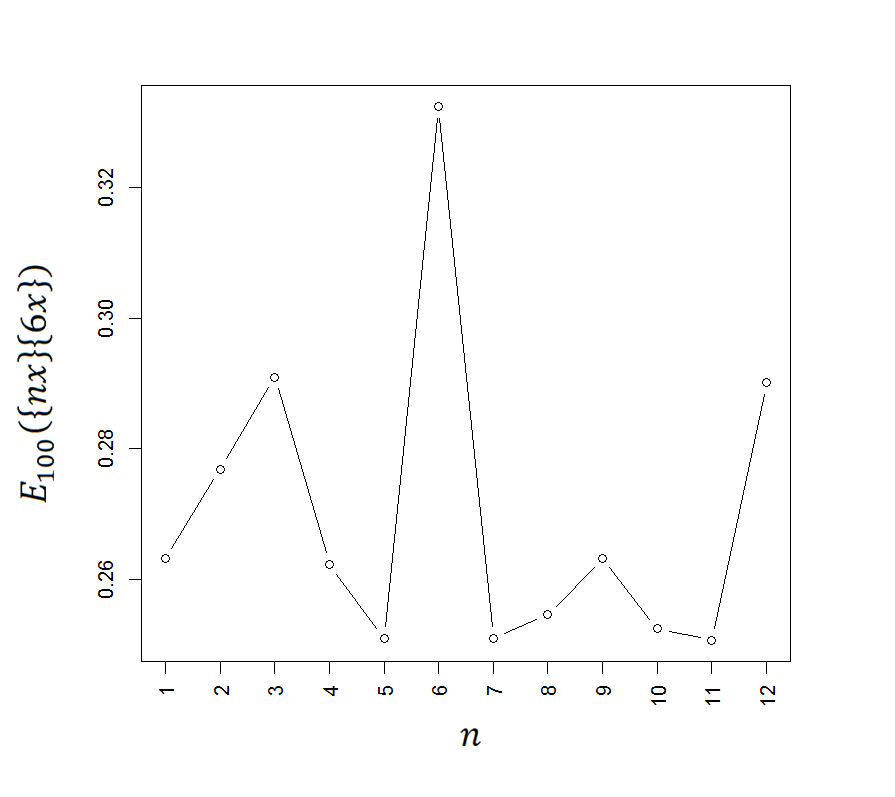}
\caption{Numerical approximation of $E_{100}\left(\{nx\}\{6x\}\right)$ for $n=1,2,...,12$. $E_{100}\left(\{nx\}\{6x\}\right)$ is a numerical approximation of the integral given by (\ref{expected value}).}
\label{fig0}
\end{figure}

The M{\"o}bius function $\mu(n)$, which is defined for $n \in \mathbb{N}$ as
\begin{equation}
\mu(n)=\begin{cases} 0; & \textnormal{$n$ has a repeated prime factor (not square-free)} \\ (-1)^{\omega(n)}; & \textnormal{$n$ has $\omega(n)$ distinct prime factors (square-free)} \end{cases} \label{mobius}
\end{equation}
is a major function in number theory. It has the important property that its sum over the divisors of any $n>1$ vanishes. That is,
\begin{equation}
\sum_{m|n}\mu(m)=0 \label{mobius vanishing}
\end{equation}
for all $n>1$. This enables the well-known technique of M{\"o}bius inversion. A related and similarly important function is the summatory M{\"o}bius function, also called the Mertens function $M(n)$, defined as 
\begin{equation}
M(n)=\sum_{m \leq n}\mu(m)=\sum_{\substack{m \nmid n \\ m<n}}\mu(m). \label{mertens} 
\end{equation}
The straightforward vanishing of the divisor sum (\ref{mobius vanishing}) contrasts strongly with the behavior of the non-divisor sum (\ref{mertens}), as the latter grows in average magnitude and fluctuates irregularly as $n$ increases. The enigmatic behavior of $M(n)$ has been linked to many important results in number theory. Landau \cite{landau} noted that the result 
\begin{equation}
M(n)= o(n) \label{mobius pnt}
\end{equation}
is equivalent to the prime number theorem (PNT) and Littlewood showed \cite{littlewood} that the stronger and unproven statement $M(n)=O\left(n^{1/2+\epsilon}\right)$ for all $\epsilon>0$ is equivalent to the famous Riemann hypothesis (RH).

In this paper we prove the following result:
\begin{thm}
\begin{equation}
\sum_{n \neq m}\frac{\mu(n)\mu(m)}{nm}E_{X}\left(\{nx\}\{mx\}\right)=-\frac{9}{2\pi^{2}}+O\left(\frac{1}{X}\right)\label{mainresult conj}
\end{equation}\label{mainthm}
\end{thm}
\hspace{-.45cm}Theorem \ref{mainthm} demonstrates a novel connection between the M{\"o}bius function $\mu(.)$ and the correlations $E_{X}\left(\{nx\}\{mx\}\right)$. Numerical evidence for (\ref{mainresult conj}) is depicted in Figure \ref{fig1}. We next present two lemmas and then present the proof of Theorem \ref{mainthm}.

\begin{figure}
\centering
\includegraphics[width=\linewidth]{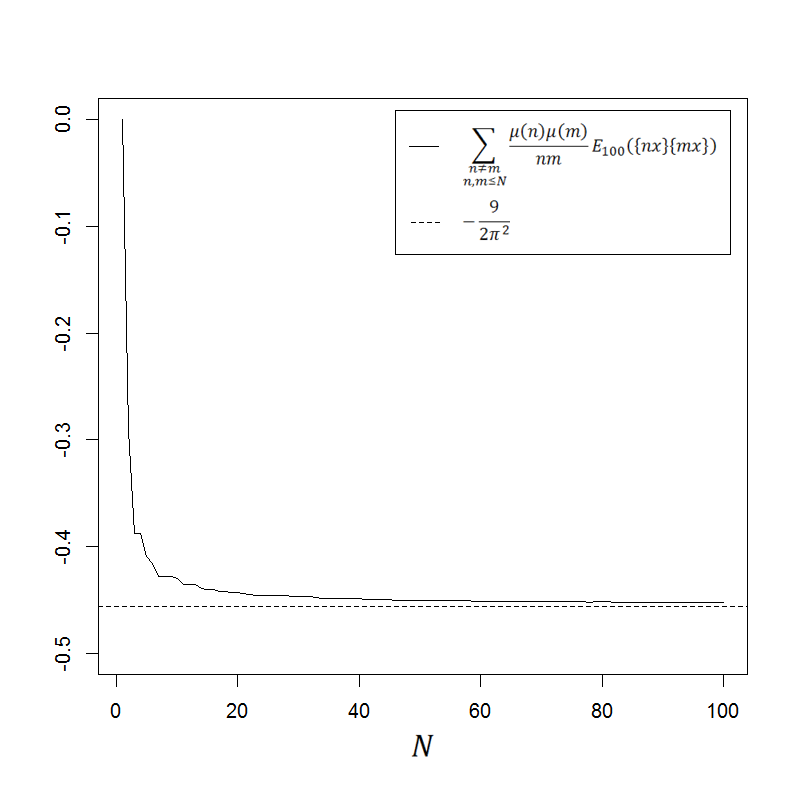}
\caption{Partial sums (Solid) corresponding to the series on (\ref{mainresult conj})'s left-hand side for $X=100$ and $1\leq N\leq 100$ and a horizontal line at the value $-\frac{9}{2\pi^{2}}$ (Dashed). $E_{100}\left(\{nx\}\{mx\}\right)$ is a numerical approximation of the integral given by (\ref{expected value}).}
\label{fig1}
\end{figure}

\section{Lemmas}
The first lemma is essentially an application of M{\"o}bius inversion and the PNT, which provides an important relationship between the sine function, the fractional part function, and the M{\"o}bius function.
\begin{lemma}
For all $x \in \mathbb{R}$,
\begin{equation}
\sin(2\pi x)=-\pi\sum_{n=1}^{\infty}\frac{\mu(n)}{n}\left \{nx\right\}. \label{sin series}
\end{equation} \label{sin series lemma}
\end{lemma}
\begin{proof}
We note the following Fourier series representation of $\left\{x\right\}$ for $x \notin \mathbb{Z}$ \cite{titchmarsh}:
\begin{equation}
\left\{x\right\}=\frac{1}{2}-\frac{1}{\pi}\sum_{k=1}^{\infty}\frac{\sin(2\pi k x)}{k} \label{frac fourier}
\end{equation}
We then apply (\ref{frac fourier}) to write 
\begin{align}
\sum_{n=1}^{\infty}\frac{\mu(n)}{n}\sum_{k=1}^{\infty}\frac{\sin(2\pi kn x)}{k}
=\sum_{r=1}^{\infty}\frac{\sin(2\pi r x)}{r}\sum_{n|r}\mu(n)=\sin(2\pi x), \label{lemma 1 some work}
\end{align}
where $r=nk$ and the last step in (\ref{lemma 1 some work}) is due to (\ref{mobius vanishing}). By (\ref{frac fourier}) and (\ref{lemma 1 some work}) we then have 
\begin{equation}
\sin(2\pi x)=\pi\sum_{n=1}^{\infty}\frac{\mu(n)}{n}\left(\frac{1}{2}-\left\{nx\right\}\right)=\frac{\pi}{2}\sum_{n=1}^{\infty}\frac{\mu(n)}{n}-\pi\sum_{n=1}^{\infty}\frac{\mu(n)}{n}\left\{nx\right\}. \label{lemma 1 last step}
\end{equation}
We then make the important note, first observed by Landau \cite{landau}, that by the PNT
\begin{equation}
\sum_{n=1}^{\infty}\frac{\mu(n)}{n}=0, \label{mob over n vanish}
\end{equation}
which, with (\ref{lemma 1 last step}), gives (\ref{sin series}). We lastly note that (\ref{sin series}) gives the correct result for $x\in \mathbb{Z}$ and thus (\ref{sin series}) holds for all $x\in \mathbb{R}$.
\end{proof}

The next lemma describes the mean square of $\{nx\}$. 
\begin{lemma}For all $n \in \mathbb{N}$,
\begin{equation}
E_{X}\left(\{n x\}^{2}\right)=\frac{1}{3}.\label{mean sq frac}
\end{equation}
\label{frac cov lemma}
\end{lemma}
\begin{proof}
We begin by writing out the expected value as
\begin{align}
E_{X}\left(\{n x\}^{2}\right)=\frac{1}{X}\int_{0}^{X}n^{2}x^{2}dx-\frac{2}{X}\int_{0}^{X}n x \lfloor n x\rfloor dx
+E_{X}\left(\lfloor n x \rfloor^{2}\right). \label{written out}
\end{align}
We next evaluate the second integral in (\ref{written out}). We make the substitution $u=n x$ to give
\begin{align}
\frac{1}{X}\int_{0}^{X}n x\lfloor n x \rfloor dx=\frac{1}{n X}\int_{0}^{n X}u \lfloor u \rfloor du \nonumber \\
=\frac{1}{n X}\left(0 \int_{0}^{1}udu+1\int_{1}^{2}udu+2\int_{2}^{3}udu+...+(n X-1)\int_{n X-1}^{n X}udu\right) \nonumber \\
=\frac{1}{n X}\sum_{k=1}^{n X-1}\frac{k\left((k+1)^{2}-k^{2}\right)}{2}=\frac{1}{n X}\sum_{k=1}^{n X-1}k^{2}+\frac{1}{n X}\sum_{k=1}^{n X-1}k. \label{written out 2}
\end{align}
We then apply the well-known Faulhaber formulae \cite{jacobi}
\begin{equation}
\sum_{k=1}^{K}k=\frac{K(K+1)}{2}
\label{faul 1}
\end{equation}
and 
\begin{equation}
\sum_{k=1}^{K}k^{2}=\frac{K(K+1)(2K+1)}{6}
\label{faul 2}
\end{equation}
in (\ref{written out 2})'s last line, simplifying to write 
\begin{equation}
\frac{1}{X}\int_{0}^{X}n x\lfloor n x \rfloor dx=\frac{(n X-1)(2n X-1)}{6}+\frac{n X-1}{4}. \label{second integral done}
\end{equation}
We note that the first integral in (\ref{written out}) is equal $n^{2} X^{2}/3$ and we then apply (\ref{second integral done}) for the second integral in (\ref{written out}), simplify, and rearrange to give
\begin{equation}
E_{X}\left(\{n x\}^{2}\right)=\frac{n^{2} X^{2}}{3}-\frac{(n X-1)(2n X-1)}{3}-\frac{n X-1}{2}+E_{X}\left(\lfloor n x \rfloor^{2}\right).
\label{frac cov}
\end{equation}
We next apply the substitution $u=n x$ and again apply (\ref{faul 2}) to write
\begin{align}
E_{X}\left(\lfloor n x \rfloor^{2}\right)=\frac{1}{X}\int_{0}^{X}\lfloor n x \rfloor^{2}dx=\frac{1}{n X}\int_{0}^{n X}\lfloor u \rfloor^{2}du \nonumber \\
=\frac{1}{n X}\left(0^{2}\int_{0}^{1}du+1^{2}\int_{1}^{2}du+2^{2}\int_{2}^{3}du...+(n X-1)^{2}\int_{n X-1}^{n X}du\right) 
\nonumber \\
=\frac{1}{n X}\sum_{k=1}^{n X-1}k^{2}=\frac{(n X-1)(2n X-1)}{6}. \label{meansqfloor}
\end{align}
We apply (\ref{meansqfloor}) in (\ref{frac cov}) to give
\begin{equation}
E_{X}\left(\{n x\}^{2}\right)=\frac{n^{2} X^{2}}{3}-\frac{(n X-1)(2n X-1)}{6}-\frac{n X-1}{2}. \label{frac cov final}
\end{equation}
Cancelling terms then gives (\ref{mean sq frac}).
\end{proof}

\section{Proof}
\begin{proof}
We first note from Lemma \ref{sin series lemma}'s result (\ref{sin series}) that, for all $x\in \mathbb{R}$,
\begin{equation}
\sin\left(2\pi x\right)+\varepsilon_{N}(x)=-\pi\sum_{n\leq N}\frac{\mu(n)}{n}\{nx\}, \label{thm1 step1}
\end{equation}
where $\varepsilon_{N}(x)\rightarrow 0$ as $N\rightarrow \infty$. We then write
\begin{equation}
\left(\sin(2\pi x)+\varepsilon_{N}(x)\right)^{2}=\sin^{2}(2\pi x)+2\sin(2\pi x)\varepsilon_{N}(x)+\varepsilon_{N}^{2}(x)=\sin^{2}(2\pi x)+\widetilde{\varepsilon}_{N}(x), \label{little bit}
\end{equation}
where
\begin{equation}
\widetilde{\varepsilon}_{N}(x)=2\sin(2\pi x)\varepsilon_{N}(x)+\varepsilon_{N}^{2}(x). \label{mod error}
\end{equation}
We then square (\ref{thm1 step1}) and apply (\ref{little bit}) to write
\begin{equation}
\sin^{2}(2\pi x)+\widetilde{\varepsilon}_{N}(x)=\pi^{2}\sum_{n\leq N}\frac{\mu^{2}(n)}{n^{2}}\{nx\}^{2}+\pi^{2}\sum_{\substack{n \neq m \\ n,m \leq N}}\frac{\mu(n)\mu(m)}{nm}\{nx\}\{mx\}.
\label{thm1 step2}
\end{equation}
We then apply the expected value to (\ref{thm1 step2}). We note that 
\begin{equation}
E_{X}\left(\sin^{2}(2\pi x)\right)=\frac{1}{2}-\frac{\sin(4\pi X)}{8\pi X}=\frac{1}{2}+O\left(\frac{1}{X}\right), \label{mean sq sine}
\end{equation}
which we apply to (\ref{thm1 step2})'s left-hand side, while on (\ref{thm1 step2})'s right-hand side we apply (\ref{mean sq frac}) from Lemma \ref{frac cov lemma} to write
\begin{equation}
\frac{1}{2}+O\left(\frac{1}{X}\right)+E_{X}\left(\widetilde{\varepsilon}_{N}(x)\right)=\frac{\pi^{2}}{3}\sum_{n\leq N}\frac{\mu^{2}(n)}{n^{2}}+\pi^{2}\sum_{\substack{n \neq m \\ n,m \leq N}}\frac{\mu(n)\mu(m)}{nm}E_{X}\left(\{nx\}\{mx\}\right). \label{thm1 step3}
\end{equation}
We then note the elementary result
\begin{equation}
\sum_{n=1}^{\infty}\frac{\left|\mu(n)\right|}{n^{s}}=\frac{\zeta(s)}{\zeta(2s)} \label{dlmf gen}
\end{equation}
for all $\textnormal{Re}(s)>1$ and hence, since $\left|\mu(n)\right|=\mu^{2}(n)$,
\begin{equation}
\sum_{n\leq N}\frac{\mu^{2}(n)}{n^{2}}=\frac{\zeta(2)}{\zeta(4)}-\sum_{n>N}\frac{\mu^{2}(n)}{n}. \label{pre dlmf 2}
\end{equation}
We next note that the density of square-free integers satisfies the well-known result \cite{gegenbauer} \cite{hardy wright}
\begin{equation}
\sum_{n \leq N}\mu^{2}(n)=\frac{N}{\zeta(2)}+O\left(\sqrt{N}\right). \label{sq free density}
\end{equation}
We then apply Abel's partial summation formula with (\ref{sq free density}) to write
\begin{align}
\sum_{n>N}\frac{\mu^{2}(n)}{n^{2}}=-\frac{N/\zeta(2)+O\left(\sqrt{N}\right)}{N^{2}}+2\int_{N}^{\infty}\frac{u/\zeta(2)+O\left(\sqrt{u}\right)}{u^{3}}du \nonumber \\ =\frac{1}{N\zeta(2)}+O\left(\frac{1}{N^{3/2}}\right),
\end{align}
which we may apply to (\ref{pre dlmf 2}) to give
\begin{equation}
\sum_{n\leq N}\frac{\mu^{2}(n)}{n^{2}}=\frac{\zeta(2)}{\zeta(4)}+O\left(\frac{1}{N}\right)=\frac{\pi^{2}/6}{\pi^{4}/90}+O\left(\frac{1}{N}\right)=\frac{15}{\pi^{2}}+O\left(\frac{1}{N}\right). \label{dlmf 2}
\end{equation}
Applying (\ref{dlmf 2}) in (\ref{thm1 step3}) and rearranging shows that 
\begin{equation}
\sum_{\substack{n \neq m \\ n,m \leq N}}\frac{\mu(n)\mu(m)}{nm}E_{X}\left(\{nx\}\{mx\}\right)=-\frac{9}{2\pi^{2}}+O\left(\frac{1}{X}\right)+O\left(\frac{1}{N}\right)+\frac{E_{X}\left(\widetilde{\varepsilon}_{N}(x)\right)}{\pi^{2}} \label{main result}
\end{equation}
as $N\rightarrow \infty$.

We next note from (\ref{thm1 step1}) that 
\begin{equation}
\varepsilon_{N}(x)=\pi\sum_{n>N}\frac{\mu(n)}{n}\{nx\}.
\label{resid}
\end{equation}
We then apply partial summation with (\ref{mertens}) to write 
\begin{align}
\sum_{n>N}\frac{\mu(n)}{n}\{nx\}=-\frac{M(N)}{N}\{Nx\}-\int_{N}^{\infty}M(u)\left(\frac{\{ux\}}{u}\right)'du \nonumber \\ 
=-\frac{M(N)}{N}\{Nx\}-\int_{N}^{\infty}\frac{M(u)}{u}\left(\{ux\}\right)'du+\int_{N}^{\infty}\frac{M(u)}{u^{2}}\{ux\}du. \label{resid partial sum}
\end{align}
We note that 
\begin{equation}
\left(\{ux\}\right)'=x-\sum_{k \in \mathbb{Z}}\delta\left(ux-k\right) \label{frac deriv}
\end{equation}
and hence 
\begin{align}
\int_{N}^{\infty}\frac{M(u)}{u}\left(\{ux\}\right)'du=x\int_{N}^{\infty}\frac{M(u)}{u}du-\sum_{k>Nx}\frac{M(k/x)}{k/x}.\label{rel int}
\end{align}
We again apply partial summation and a change of variable to write
\begin{align}
\sum_{k>Nx}\frac{M\left(k/x\right)}{k/x}=-\frac{M(N)}{N}-\int_{Nx}^{\infty}M\left(\frac{u}{x}\right)\left(\frac{x}{u}\right)'du \nonumber \\ =-\frac{M(N)}{N}+\int_{Nx}^{\infty}M\left(\frac{u}{x}\right)\frac{x}{u^{2}}du=-\frac{M(N)}{N}+\int_{N}^{\infty}\frac{M(u)}{u^{2}}du. \label{weird sum}
\end{align}
Applying (\ref{weird sum}) in (\ref{rel int}) and then applying the result in (\ref{resid partial sum}) thus gives
\begin{equation}
\sum_{n>N}\frac{\mu(n)}{n}\{nx\}=-\frac{M(N)}{N}\left(1+\{Nx\}\right)+\int_{N}^{\infty}\frac{M(u)}{u^{2}}\left(1+\{ux\}\right)du-x\int_{N}^{\infty}\frac{M(u)}{u}du. \label{resid partial sum finalish}
\end{equation}

We next note the classical result (see e.g. \cite{allison})
\begin{equation}
M(u)=O\left(\frac{u}{e^{c\sqrt{\log u}}}\right) 
\label{classic}
\end{equation}
for $c>0$, from which we have $M(N)/N=O\left(e^{-c\sqrt{\log u}}\right)$ and 
\begin{equation}
\int_{N}^{\infty}\frac{M(u)}{u^{2}}\left(1+\{ux\}\right)du=O\left(\int_{N}^{\infty}\frac{du}{ue^{c\sqrt{\log u}}}\right)=O\left(\frac{\sqrt{\log N}}{e^{c\sqrt{\log N}}}\right), \label{middle term}
\end{equation}
which we then apply in (\ref{resid partial sum finalish}) to write 
\begin{equation}
\sum_{n>N}\frac{\mu(n)}{n}\{nx\}=O\left(\frac{\sqrt{\log N}}{e^{c\sqrt{\log N}}}\right)-x\int_{N}^{\infty}\frac{M(u)}{u}du. \label{resid partial sum finaler}
\end{equation}
Since (\ref{resid partial sum finaler}) vanishes as $N\rightarrow \infty$ by (\ref{sin series}), we conclude that 
\begin{equation}
\int_{N}^{\infty}\frac{M(u)}{u}du\rightarrow 0 \label{nice}
\end{equation}
as $N\rightarrow \infty$. Then by (\ref{mod error}), (\ref{resid}), and (\ref{resid partial sum finaler}) we have 
\begin{align}
\widetilde{\varepsilon}_{N}(x)=2\pi\sin(2\pi x)\left(O\left(\frac{\sqrt{\log N}}{e^{c\sqrt{\log N}}}\right)-x\int_{N}^{\infty}\frac{M(u)}{u}du\right) \nonumber \\ +O\left(\frac{\log N}{e^{2c\sqrt{\log N}}}\right)+\pi^{2}x^{2}\left(\int_{N}^{\infty}\frac{M(u)}{u}du\right)^{2} +  xO\left(\frac{\sqrt{\log N}}{e^{c\sqrt{\log N}}}\right)\int_{N}^{\infty}\frac{M(u)}{u}du. \label{mod error long}
\end{align}
We then apply the expected value, noting that 
\begin{equation}
E_{X}\left(2\pi \sin\left(2\pi x\right)\right)=\frac{1-\cos\left(2\pi X\right)}{X}=O\left(\frac{1}{X}\right)
\end{equation}
and
\begin{equation}
E_{X}\left(2\pi x\sin\left(2\pi x\right)\right)=\frac{\sin(2\pi X)}{2\pi X}-\cos(2\pi X)=O\left(\frac{1}{X}\right)-\cos(2\pi X)
\end{equation}
to give 
\begin{align}
E_{X}\left(\widetilde{\varepsilon}_{N}(x)\right)=O\left(\frac{1}{X}\right)O\left(\frac{\sqrt{\log N}}{e^{c\sqrt{\log N}}}\right)+\left(\cos\left(2\pi X\right)+O\left(\frac{1}{X}\right)\right)\int_{N}^{\infty}\frac{M(u)}{u}du \nonumber \\ +O\left(\frac{\log N}{e^{2c\sqrt{\log N}}}\right)+\frac{\pi^{2}X^{2}}{3}\left(\int_{N}^{\infty}\frac{M(u)}{u}du\right)^{2}\nonumber \\ +\frac{X}{2}O\left(\frac{\sqrt{\log N}}{e^{c\sqrt{\log N}}}\right)\int_{N}^{\infty}\frac{M(u)}{u}du. \label{mod error very long}
\end{align}
Applying (\ref{mod error very long}) in (\ref{main result}) and collecting terms gives
\begin{align}
\sum_{\substack{n \neq m \\ n,m \leq N}}\frac{\mu(n)\mu(m)}{nm}E_{X}\left(\{nx\}\{mx\}\right)=-\frac{9}{2\pi^{2}}+O\left(\frac{\log N}{e^{2c\sqrt{\log N}}}\right)\nonumber \\ +O\left(\frac{1}{X}\right)\left(1+O\left(\frac{\sqrt{\log N}}{e^{c\sqrt{\log N}}}\right)+\int_{N}^{\infty}\frac{M(u)}{u}du\right) \nonumber \\+\left(\cos\left(2\pi X\right)+O\left(\frac{X\sqrt{\log N}}{e^{c\sqrt{\log N}}}\right)\right)\int_{N}^{\infty}\frac{M(u)}{u}du+\frac{X^{2}}{3}\left(\int_{N}^{\infty}\frac{M(u)}{u}du\right)^{2}. \label{main result long}
\end{align}
Taking $N\rightarrow \infty$ and applying (\ref{nice}) in (\ref{main result long}) then completes the proof.
\end{proof}

\section{Acknowledgements}
The author thanks Altan Allawala and Christophe Vignat for their helpful comments.

\end{document}